\documentclass[a4paper,oneside,11pt]{article}%
\usepackage{makeidx}
\usepackage[english]{babel}
\usepackage{amsmath}
\usepackage{amsfonts}
\usepackage{amssymb}
\usepackage{stmaryrd}
\usepackage{graphicx}
\usepackage{mathrsfs}
\usepackage[colorlinks,linkcolor=red,anchorcolor=blue,citecolor=blue,urlcolor=blue]{hyperref}
\usepackage[symbol*,ragged]{footmisc}
\providecommand{\U}[1]{\protect\rule{.1in}{.1in}}
\providecommand{\U}[1]{\protect \rule{.1in}{.1in}}

\newtheorem{theorem}{Theorem}[section]

\newtheorem{corollary}[theorem]{Corollary}

\newtheorem{lemma}[theorem]{Lemma}

\newenvironment{proof}[1][Proof]{\noindent \textbf{#1.} }{\  \rule{0.5em}{0.5em}}
\numberwithin{equation}{section}

\begin{document}

\title{A Remark on the Piatetski--Shapiro--Hua Theorem}

\author{Jinjiang Li\footnotemark[1]\,\,\,\, \, \& \,\,Min Zhang\footnotemark[2] \vspace*{-4mm} \\
$\textrm{\small Department of Mathematics, China University of Mining and Technology}^{*\,\dag}$
                    \vspace*{-4mm} \\
     \small  Beijing 100083, P. R. China  }

\footnotetext[2]{Corresponding author. \\
    \quad\,\, \textit{ E-mail addresses}: \href{mailto:jinjiang.li.math@gmail.com}{jinjiang.li.math@gmail.com} (J. Li),
     \href{mailto:min.zhang.math@gmail.com}{min.zhang.math@gmail.com} (M. Zhang).    }

\date{}
\maketitle


{\textbf{Abstract}}: In this paper, we prove that for any fixed $205/243<\gamma\leqslant1$, every sufficiently large $N$ satisfying $N\equiv 5 \pmod {24}$ can be represented as five squares of primes with one prime in $\mathcal{P}_\gamma$, which improves the previous result of Zhang and Zhai.


{\textbf{MR(2010) Subject Classification}}: 11L20, 11P05, 11P32.

\section{Introduction and main result}

Let $k\geqslant1$ be a positive integer, $s\geqslant2^k+1$ and $N$ be a sufficiently large integer. Let $T(N;s,k)$ denote the number of
    solutions of the equation
    \begin{equation}\label{representation}
        N=p_1^k+p_2^k+\cdots+p_s^k.
    \end{equation}

    In 1937, Vinogradov  \cite{Vinogradov} proved the well known Goldbach--Vinogradov theorem: Every sufficiently large odd integer $N$ can be represented
    as a sum of three primes. It can be stated in a more exactly quantitative form, i.e.,
    \begin{equation*}
       T(N;3,1)=\frac{1}{2}\mathfrak{S}(N)\frac{N^2}{\log^3N}+O\bigg(\frac{N^2}{\log^A N}\bigg)
    \end{equation*}
    for any $A>4$, where $\mathfrak{S}(N)$ denotes the singular series
    \begin{equation*}
       \mathfrak{S}(N)=\prod_{p|N}\bigg(1-\frac{1}{(p-1)^2}\bigg)\prod_{p\nmid N}\bigg(1+\frac{1}{(p-1)^3}\bigg).
    \end{equation*}
    For the general case $k\geqslant2$, Hua~\cite{Hua-book} proved that there holds
    \begin{equation*}
       T(N;s,k)=\mathfrak{S}_{k,s}(N)\frac{\Gamma^s(1/k)}{\Gamma(s/k)}\frac{N^{s/k-1}}{\log^s N}+O\bigg(\frac{N^{s/k-1}\log\log N}{\log^{s+1}N}\bigg),
    \end{equation*}
    where $\mathfrak{S}_{k,s}(N)$ is the singular series related to the representation (\ref{representation}).

    Motivated by earlier work of Erd\H{o}s and Nathanson \cite{Erdos-Nathanson} on sums of squares,
some mathematicians considered the question of whether one could find thin subsets of primes which were still sufficient to obtain the  Goldbach--Vinogradov theorem. In 1986, based on probability considerations, Wirsing \cite{Wirsing} proved that there exists a subset $\mathcal{S}$ of primes with the property
\begin{equation*}
     \sum_{\substack{p\leqslant x\\p\in\mathcal{S}}}1\ll(x\log x)^{1/3},
\end{equation*}
which serves this purpose. Although Wirsing's result is best possible apart from the logarithmic factor, it does not lead to a subset of primes which is
constructive or recognizable.

 Primes of the form $[n^c]$, where $1\leqslant c<2$, are called Piatetski--Shapiro primes. Let $\gamma=1/c$, so that the set of the Piatetski--Shapiro
 primes of type
  $$\mathcal{P}_\gamma=\big\{p:p=[n^{1/\gamma}]\,\,\textrm{for some $n\in\mathbb{N}$}\big\}$$
  is a well--known thin set of prime numbers. Piatetski--Shapiro first proved that the asymptotic formula
  \begin{equation}\label{shapiro-theorem}
      \sum_{\substack{n\leqslant x\\ [n^{c}]=p}}1=\big(1+o(1)\big)\frac{x}{c\log x}
  \end{equation}
  holds for $1<c<12/11$. Since then, this range for $c$ has been improved by a number of authors. The best results are given by 
  Rivat and Sargos~\cite{Rivat-Sargos}, where it is proved that (\ref{shapiro-theorem}) holds for $1<c<2817/2426$.

  In 1998, Zhai~\cite{Zhai} obtained a theorem which had two interesting corollaries:
\begin{corollary}
   For any fixed $43/44<\gamma\leqslant1$, every sufficiently large $N$ satisfying $N\equiv 5 \pmod {24}$ can be represented as five squares of primes
   with each prime in $\mathcal{P}_\gamma$.
\end{corollary}
\begin{corollary}
   For any fixed $27/29<\gamma\leqslant1$, every sufficiently large $N$ satisfying $N\equiv 5 \pmod {24}$ can be represented as five squares of primes
   with one prime in $\mathcal{P}_\gamma$.
\end{corollary}
 In 2005, Zhang and Zhai~\cite{Zhang-Zhai} improved the theorem and the two corollaries of Zhai~\cite{Zhai} and obtained the two following corollaries:
\begin{corollary}
   For any fixed $249/256<\gamma\leqslant1$, every sufficiently large $N$ satisfying $N\equiv 5 \pmod {24}$ can be represented as five squares of primes
   with each prime in $\mathcal{P}_\gamma$.
\end{corollary}
\begin{corollary}\label{improve-corollary}
   For any fixed $75/82<\gamma\leqslant1$, every sufficiently large $N$ satisfying $N\equiv 5 \pmod {24}$ can be represented as five squares of primes
   with one prime in $\mathcal{P}_\gamma$.
\end{corollary}

The purpose of this paper is to present an approach different from that of Zhai~\cite{Zhai} and Zhang and Zhai~\cite{Zhang-Zhai} which leads to
an improvement of Corollary~\ref{improve-corollary}.

\begin{theorem}\label{theorem}
   For any fixed $205/243<\gamma\leqslant1$, every sufficiently large $N$ satisfying $N\equiv 5 \pmod {24}$ can be represented as five squares of primes
   with one prime in $\mathcal{P}_\gamma$.
\end{theorem}

\section{Preliminary Lemmas}

In order to prove Theorem~\ref{theorem} we need the following two lemmas.

\begin{lemma}\label{Rivat-Wu-lemma}
   For any fixed $205/243<\gamma\leqslant1$, we have
   \begin{equation*}
     P_{\gamma}(x):=\sum_{\substack{x<p\leqslant2x\\p=[n^{1/\gamma}]}}1\gg\frac{x^\gamma}{\gamma \log x}.
   \end{equation*}
\end{lemma}
\begin{proof}
  See Theorem 1 of Rivat and Wu~\cite{Rivat-Wu}.
\end{proof}

\begin{lemma}\label{Harman-Kumchev-lemma}
   Let $E(x)$ denote the number of positive integers $n$ not exceeding $x$ and satisfying $n\equiv4\pmod {24}$ which can not be
   represented as a sum of four squares of primes. Then for any $\varepsilon>0$, we have
   \begin{equation*}
      E(x)\ll x^{7/20+\varepsilon}.
   \end{equation*}
\end{lemma}
\begin{proof}
  See Theorem 2 of Harman and Kumchev~\cite{Harman-Kunchev}.
\end{proof}

\section{Proof of Theorem~\ref{theorem}}
Let
\begin{equation*}
   \mathcal{A}=\left\{N-p^2:\,\sqrt{\frac{N}{5}}\leqslant p\leqslant \sqrt{\frac{4N}{5}},\,p\in\mathcal{P}_\gamma\right\}.
\end{equation*}
Then, by Lemma~\ref{Rivat-Wu-lemma}, we have
\begin{equation*}
   |\mathcal{A}|\gg \frac{N^{\gamma/2}}{\gamma\log N}.
\end{equation*}

Let $E(\mathcal{A})$ denote the set of integers in $\mathcal{A}$ which can not be represented as a sum of
four squares of primes. Then, by Lemma~\ref{Harman-Kumchev-lemma}, we obtain
\begin{equation*}
   |E(\mathcal{A})|\ll N^{7\gamma/40+\varepsilon},
\end{equation*}
and
\begin{equation*}
   |\mathcal{A} \setminus E(\mathcal{A})|\gg \frac{N^{\gamma/2}}{\gamma\log N}.
\end{equation*}

For any $N-p^2\in\mathcal{A}\setminus E(\mathcal{A})$, there exist four primes $p_1,\,p_2,\,p_3,\,p_4$ such that
\begin{equation*}
   N-p^2=p_1^2+p_2^2+p_3^2+p_4^2.
\end{equation*}
Therefore, we get
\begin{equation*}
   N=p^2+p_1^2+p_2^2+p_3^2+p_4^2,\quad \sqrt{\frac{N}{5}}\leqslant p\leqslant \sqrt{\frac{4N}{5}},\quad p\in\mathcal{P}_\gamma,
\end{equation*}
and Theorem \ref{theorem} follows.

\bigskip
\bigskip

\textbf{Acknowledgement}

   The authors would like to express the most and the greatest sincere gratitude to Professor Wenguang Zhai for his valuable
advice and constant encouragement.

\end{document}